\documentclass[11pt]{amsart}
\usepackage{amssymb,amsmath,amsthm,latexsym,stmaryrd,xcolor}

\usepackage[T1]{fontenc}
\usepackage{lmodern}

\usepackage[pagebackref,colorlinks=true]{hyperref}  
\newtheorem{theorem}{Theorem}[section]

\newtheorem{definition}[theorem]{Definition}

\newcommand{\f}{\varphi}
\newcommand{\g}{\tilde{g}}
\newcommand{\bg}{\bar{g}}
\newcommand{\n}{\nabla}

\newcommand{\M}{(M,\A\f,\A\xi,\A\eta,\A{}g)}

\newcommand{\I}{\iota}

\newcommand{\R}{\mathbb R}

\newcommand{\X}{\mathfrak X}
\newcommand{\F}{\mathcal{F}}
\newcommand{\LL}{\mathcal{L}}
\newcommand{\ta}{\theta}
\newcommand{\om}{\omega}
\newcommand{\lm}{\lambda}
\newcommand{\sm}{\sigma}

\newcommand{\al}{\alpha}
\newcommand{\bt}{\beta}

\newcommand{\D}{\mathrm{d}}

\DeclareMathOperator{\Span}{span} 
\DeclareMathOperator{\im}{im} 

\newcommand{\ddu}[1]{\dfrac{\partial u}{\partial x^{#1}}}
\newcommand{\ddv}[1]{\dfrac{\partial v}{\partial x^{#1}}}
\newcommand{\ddw}[1]{\dfrac{\partial w}{\partial x^{#1}}}
\newcommand{\p}{\partial}

\newcommand{\thmref}[1]{Theorem~\ref{#1}}

\newcommand{\A}{\allowbreak{}}

\begin{document}

\title[Yamabe solitons on conformal Sasaki-like ...]
{Yamabe solitons on conformal Sasaki-like almost contact B-metric manifolds}

\author[M. Manev]{Mancho Manev
}

\address[M. Manev]{
University of Plovdiv Paisii Hilendarski,
Faculty of Mathematics and Informatics,
Department of Algebra and Geometry,
24 Tzar Asen St.,
Plovdiv 4000, Bulgaria
\&
Medical University -- Plovdiv,
Faculty of Pharmacy,
Department of Medical Physics and Biophysics,
15A Vasil Aprilov Blvd,
Plovdiv 4002, Bulgaria}
\email{mmanev@uni-plovdiv.bg}

\begin{abstract}
A Yamabe soliton is defined on arbitrary almost contact B-metric manifold, which is obtained by a contact conformal transformation of the Reeb vector field, its dual contact 1-form, the B-metric, and its associated B-metric.
The cases when the given manifold is cosymplectic or Sasaki-like are studied.
In this way, manifolds from one of the main classes of the studied manifolds are obtained. The same class contains the conformally equivalent manifolds of cosymplectic manifolds by the usual conformal transformation of the B-metric.
An explicit 5-dimensional example of a Lie group is given, which is characterized in relation to the obtained results.
\end{abstract}


%



\maketitle

\section*{Introduction}

Richard Hamilton introduced the notion of the Yamabe flow in 1988 (see \cite{Ham88}) as an apparatus for constructing  metrics with constant scalar curvature. The significance of this problem from the point of view of mathematical physics is that the Yamabe flow corresponds to the case of fast diffusion of the porous medium equation \cite{ChLuNi}.

A self-similar solution of the Yamabe flow, defined on a Riemannian manifold or a pseudo-Riemannian manifold, is called a Yamabe soliton and is determined by (\cite{BarRib})
\begin{equation*}\label{YS-intro}
  \frac12 \LL_{v} g = (\tau - \sm) g,
\end{equation*}
where $\LL_{v} g$ stands for the Lie derivative of the metric $g$ along the vector field $v$, $\tau$ denotes the scalar curvature of $g$ and $\sm$ is a constant.
Many authors have studied the Yamabe soliton on different types of manifolds since the introduction of this concept (see e.g. \cite{Cao1,Chen1,Das1,Gho1,Roy1}).

In the present paper we begin the study of the mentioned Yamabe solitons on almost contact B-metric manifolds. The geometry of these manifolds is significantly influenced by the presence of two B-metrics, which are interconnected by the almost contact structure.

It is a well-known fact that the Yamabe flow preserves the conformal class of the metric. 
So this gives us a reason to study Yamabe solitons and conformal transformations together. 
Contact conformal transformations were studied in \cite{Man4}, which transform not only the metric but also the Reeb vector field and its associated contact 1-form through the pair of B-metrics.
According to this work, the class of almost contact B-metric manifolds, 
which is closed under the action of contact conformal transformations, 
is the direct sum of the four main classes among the eleven basic classes of these manifolds
known from the classification of Ganchev-Mihova-Gribachev in \cite{GaMiGr}. 
Main classes are called those whose manifolds are characterized by the fact that the covariant derivative of the structure tensors with respect to the Levi-Civita connection of some of the B-metrics is expressed only by the pair of B-metrics and the corresponding traces.

We study Yamabe solitons on two of the simplest types of manifolds among those studied, namely cosymplectic and Sasaki-like.
The former have parallel structure tensors with respect to the Levi-Civita connections of the B-metrics.
The latter are those whose complex cone is a K\"ahler manifold with a pair of Norden metrics, i.e. again with parallel structure tensors with respect to the Levi-Civita connections of the metrics.
Note that the class of Sasaki-like manifolds does not contain cosymplectic manifolds, although they are in each of the eleven basic classes of the classification used.

We find that the manifolds thus constructed in both cases belong to one of the main classes, the only one which contains the conformally equivalent manifolds of the cosymplectic ones by the usual conformal transformations.

The present paper is organized as follows. 
Section 1 is devoted to the basic concepts of almost contact B-metric manifolds, contact conformal transformations of the structure tensors on them and the introduction of the notion of a Yamabe soliton 
on a transformed almost contact B-metric manifold.
In Section 2 and Section 3, we study the constructed manifolds when the initial manifold is cosymplectic and Sasaki-like, respectively.
In Section 4, we provide an explicit example of a Lie group as
a 5-dimensional manifold equipped with the structures studied.

\section{Almost contact B-metric manifolds, contact conformal transformations and Yamabe solitons}

Let $(M,\f,\xi,\eta,g)$ be an almost contact manifold with
B-metric or an \emph{almost contact B-metric manifold}, i.e. $M$ is
a $(2n+1)$-dimensional differen\-tia\-ble manifold with an almost
contact structure $(\f,\xi,\eta)$ consisting of an endomorphism
$\f$ of the tangent bundle, a vector field $\xi$, its dual 1-form
$\eta$ as well as $M$ is equipped with a pseudo-Riemannian metric
$g$  of signature $(n+1,n)$, such that the following algebraic
relations are satisfied \cite{GaMiGr}
\begin{gather}
\f\xi = 0,\quad \f^2 = -\I + \eta \otimes \xi,\quad
\eta\circ\f=0,\quad \eta(\xi)=1, \label{str}\\%
g(X, Y ) = - g(\f X, \f Y ) + \eta(X)\eta(Y) \label{g}
\end{gather}
for arbitrary $X$, $Y$ of the algebra $\X(M)$ on the smooth vector
fields on $M$, where $\I$ stands for the identity transformation on $\X(M)$.

The manifolds $(M,\f,\xi,\eta,g)$ are also known as \emph{almost contact complex Rie\-man\-ni\-an
manifolds} (see e.g. \cite{IvMaMa45}).

Further, $X$, $Y$, $Z$ will stand for arbitrary elements of
$\X(M)$ or vectors in the tangent space $T_pM$ of $M$ at an arbitrary
point $p$ in $M$.

The associated B-metric $\tilde{g}$ of $g$ on $M$ is defined by
\begin{equation*}\label{tg}
\tilde{g}(X,Y)=g(X,\f Y)\allowbreak+\eta(X)\eta(Y)
\end{equation*}
and it is also of signature $(n+1,n)$ as its counterpart $g$.

The fundamental tensor $F$ of type (0,3) on $\M$ is defined by
\begin{equation*}\label{F=nfi}
F(X,Y,Z)=g\bigl( \left( \nabla_X \f \right)Y,Z\bigr),
\end{equation*}
where $\n$ is the Levi-Civita connection of $g$.
The following properties of $F$ are consequences of \eqref{str} and \eqref{g}:
\begin{equation*}\label{F-prop}
\begin{split}
F(X,Y,Z)&=F(X,Z,Y)\\[4pt]
&=F(X,\f Y,\f Z)+\eta(Y)F(X,\xi,Z) +\eta(Z)F(X,Y,\xi)
\end{split}
\end{equation*}
and relations of $F$ with $\n\xi$ and $\n\eta$ are:
\begin{equation}\label{FXieta}
    \left(\n_X\eta\right)Y=g\left(\n_X\xi,Y\right)=F(X,\f Y,\xi).
\end{equation}

The following 1-forms, known as Lee forms of the manifold, are associated with $F$:
\begin{equation*}\label{tataom}
    \ta=g^{ij}F(E_i,E_j,\cdot),\qquad \ta^*=g^{ij}F(E_i,\f E_j,\cdot),\qquad \om=F(\xi,\xi,\cdot),
\end{equation*}
where $g^{ij}$ are the components of the inverse matrix of $g$
with respect to a basis $\left\{E_i;\xi\right\}$
$(i=1,2,\dots,2n)$ of $T_pM$.
The following general identities for the Lee forms of $\M$ are known from \cite{ManGri1}
\begin{equation}\label{tataom=id}
    \ta^*\circ \f=-\ta\circ\f^2,\qquad \om(\xi)=0.
\end{equation}

A classification of the almost contact B-metric manifolds
in terms of $F$ is given in \cite{GaMiGr}. This classification
includes eleven basic classes $\F_1$, $\F_2$, $\dots$, $\F_{11}$.
Their intersection is the special class $\F_0$ defined by
condition for the vanishing of $F$. Hence $\F_0$ is the class of cosymplectic B-metric manifolds,
where the structures $\f$, $\xi$, $\eta$, $g$, $\tilde{g}$ are $\n$-parallel.

Using the pair of B-metrics $g$ and $\g$ as well as $\eta\otimes\eta$, in \cite{ManGri1}, the author and K. Gribachev introduced the so-called contact conformal transformation of the B-metric $g$ into a new B-metric $\bar g$ for $(M,\f,\xi,\eta,g)$. Later, in \cite{Man4}, this transformation is generalized as a contact conformal transformation that gives an almost contact B-metric structure $(\f,\bar\xi,\bar\eta,\bar g)$ as follows
\begin{equation}\label{cct}
\begin{array}{l}
\bar\xi=e^{-w}\xi,\qquad \bar\eta=e^{w}\eta,\\[4pt]
\bar g= e^{2u}\cos{2v}\, g + e^{2u}\sin{2v}\, \g + \left(e^{2w}-e^{2u}\cos{2v}-e^{2u}\sin{2v}\right)\eta\otimes\eta,
\end{array}
\end{equation}
where $u, v, w$ are differentiable functions on $M$.

The corresponding tensors $\bar F$ and $F$ are related by \cite{Man4}, see also \cite[(22)]{ManIv38},
\begin{equation}\label{ff}
\begin{aligned}
    2\bar{F}&(X,Y,Z)= 2e^{2u}\cos{2v}\, F(X,Y,Z)
    \\[4pt]
    &
    +e^{2u}\sin{2v} \left[F(\f Y,Z,X)-F(Y,\f Z,X)+F(X,\f Y,\xi)\eta(Z)\right]\\[4pt]
    &+e^{2u}\sin{2v}
    \left[F(\f Z,Y,X)-F(Z,\f Y,X)+F(X,\f
    Z,\xi)\eta(Y)\right]\\[4pt]
    &+(e^{2w}-e^{2u}\cos{2v})\left[F(X,Y,\xi)+F(\f Y,\f
    X,\xi)\right]\eta(Z)\\[4pt]
    &+(e^{2w}-e^{2u}\cos{2v})
    \left[F(X,Z,\xi)+F(\f Z,\f X,\xi)\right]\eta(Y)\\[4pt]
    &+(e^{2w}-e^{2u}\cos{2v})
    \left[F(Y,Z,\xi)+F(\f Z,\f Y,\xi)\right]\eta(X)\\[4pt]
    &+(e^{2w}-e^{2u}\cos{2v})\left[F(Z,Y,\xi)+F(\f Y,\f Z,\xi)\right]\eta(X)
\\[4pt]
    &-2e^{2u}\left[
    \cos{2v}\,\al(Z)
    +\sin{2v}\,\bt(Z)\right]g(\f X,\f Y)
\\[4pt]
    &-2e^{2u}\left[
    \cos{2v}\,\al(Y)
    +\sin{2v}\,\bt(Y)\right]g(\f X,\f Z)
\\[4pt]
    &-2e^{2u}\left[
    \cos{2v}\,\bt(Z)
    -\sin{2v}\,\al(Z)\right]g(X,\f Y)
\\[4pt]
    &-2e^{2u}\left[
    \cos{2v}\,\bt(Y)
    -\sin{2v}\,\al(Y)\right]g(X,\f Z)
\\[4pt]
    &+2e^{2w}\eta(X)\left[\eta(Y)\D w(\f Z)+\eta(Z)\D w(\f
    Y)\right],
\end{aligned}
\end{equation}
where we use the following notations for  brevity
\begin{equation}\label{albt}
\al = \D u\circ \f + \D v, \qquad \bt = \D u - \D v\circ \f.
\end{equation}

In the general case, the relations between the Lee forms of the manifolds $\M$ and $(M,\f,\bar\xi,\bar\eta,\bar g)$ are the following (see \cite{Man4})
\begin{equation}\label{ttbartt}
    \bar{\ta} = \ta+2n\, \al,\qquad
    \bar{\ta}^* = \ta^* +2n\,\bt,\qquad
    \bar{\om}  = \om +\D w\circ \f.
\end{equation}

\begin{definition}
We say that the B-metric $\bg$ generates a Yamabe soliton with potential the Reeb vector field $\bar\xi$ and soliton constant $\bar\sigma$ on a conformal almost contact B-metric manifold $(M,\f,\bar\xi,\bar\eta,\bg)$, if the following condition is satisfied
\begin{equation}\label{YS}
  \frac12 \LL_{\bar\xi} \bg = (\bar\tau - \bar\sm)\bar g,
\end{equation}
where $\bar\tau$ is the scalar curvature of $\bar g$.
\end{definition}

It is well known the following expression of the Lie derivative in terms of the covariant derivative with respect to the Levi-Civita connection $\bar\n$ of 
$\bg$
\begin{equation}\label{L1}
  \left(\LL_{\bar\xi} \bg\right)(X,Y) = \bg\left(\bar\n_X \bar\xi, Y\right)+\bg\left(X, \bar\n_Y \bar\xi\right).
\end{equation}
%

\section{The case when the given manifold is cosymplectic}

In this section, we consider $\M$ to be a cosymplectic manifold with almost contact B-metric structure, i.e. an $\F_0$-manifold, defined by covariant constant structure tensor $\f$ (and consequently the same condition for $\xi$, $\eta$, $\g$ is valid) with respect to
$\n$ of $g$.
Therefore $F=0$ and \eqref{ff} takes the form
\begin{equation}\label{ff0}
\begin{aligned}
    \bar{F}(X,Y,Z)=& -e^{2u}\left[
    \cos{2v}\,\al(Z)
    +\sin{2v}\,\bt(Z)\right]g(\f X,\f Y)
\\[4pt]
    &-e^{2u}\left[
    \cos{2v}\,\al(Y)
    +\sin{2v}\,\bt(Y)\right]g(\f X,\f Z)
\\[4pt]
    &-e^{2u}\left[
    \cos{2v}\,\bt(Z)
    -\sin{2v}\,\al(Z)\right]g(X,\f Y)
\\[4pt]
    &-e^{2u}\left[
    \cos{2v}\,\bt(Y)
    -\sin{2v}\,\al(Y)\right]g(X,\f Z)
\\[4pt]
    &+e^{2w}\eta(X)\left[\eta(Y)\D w(\f Z)+\eta(Z)\D w(\f
    Y)\right].
\end{aligned}
\end{equation}

From the latter equality, bearing in mind the general identity \eqref{FXieta} for the manifolds under consideration,
we obtain
\begin{equation*}\label{xi0}
\begin{aligned}
    \bar g\left(\bar\n_X \bar\xi, Y\right)=& -e^{2u-w}\left[\cos{2v}\,\bt(\xi)-\sin{2v}\,\al(\xi)\right]g(\f X,\f Y)
\\[4pt]
    & +e^{2u-w}\left[\cos{2v}\,\al(\xi)+\sin{2v}\,\bt(\xi)\right]g( X,\f Y)
\\[4pt]
    &+e^{w}\eta(X)\D w(\f^2 Y).
\end{aligned}
\end{equation*}

Combining the last expression with \eqref{albt} and \eqref{L1} gives the following
\begin{equation}\label{Lxi0}
\begin{aligned}
    \left(\LL_{\bar\xi} \bg\right)(X,Y) =& -2e^{2u-w}\left[\cos{2v}\,\D{u}(\xi)-\sin{2v}\,\D{v}(\xi)\right]g(\f X,\f Y)
\\[4pt]
    & +2e^{2u-w}\left[\cos{2v}\,\D{v}(\xi)+\sin{2v}\,\D{u}(\xi)\right]g( X,\f Y)
\\[4pt]
    &+e^{w}\left[\eta(X)\D w(\f^2 Y)+\eta(Y)\D w(\f^2 X)\right].
\end{aligned}
\end{equation}

\begin{theorem}\label{thm:F0-YS}
An almost contact B-metric manifold that is cosymplectic can be transformed by a contact conformal transformation of type \eqref{cct} so that the transformed B-metric is a Yamabe soliton with potential the transformed Reeb vector field
and a soliton constant $\bar\sm$
if and only if the functions $(u,v,w)$ of the used trans\-formation satisfy the conditions
\begin{equation}\label{F0-YS-uvw}
    \D u(\xi)=0,\qquad \D v(\xi)=0,\qquad \D w=\D w(\xi)\eta.
\end{equation}
Moreover, the obtained Yamabe soliton has a constant scalar curvature with value $\bar\tau=\bar\sm$ and the obtained almost contact B-metric manifold belongs to the subclass of the main class $\F_1$ determined by the conditions:
\begin{equation}\label{F0F1-YS}
    \bar{\ta} = 2n \left\{\D u\circ \f - \D v\circ \f^2\right\},
    \qquad
    \bar{\ta}^* = -2n \left\{\D u\circ \f^2 + \D v\circ \f\right\}.
\end{equation}
\end{theorem}
\begin{proof}
Is we assume that $\bg$ generates a Yamabe soliton with potential $\bar\xi$ and a soliton constant $\bar\sigma$ on $(M,\f,\bar\xi,\bar\eta,\bg)$, then due to \eqref{YS} and \eqref{Lxi0} we have the following
\begin{equation}\label{L-YS}
\begin{split}
&\ 
    -e^{2u-w}
  \left\{\left[\cos{2v}\,\D{u}(\xi)-\sin{2v}\,\D v(\xi)\right]g(\f X,\f Y)\right.\\[4pt]
  &\phantom{\hspace{48pt}} \left.
  -\left[\sin{2v}\,\D{u}(\xi)+\cos{2v}\,\D v(\xi)\right]g(X,\f Y)\right\}\\[4pt]
  &
  +\frac12 e^{w}\left\{\D w(\f^2 X)\eta(Y) +\D w(\f^2 Y)\eta(X)\right\}\\[4pt]
  &=(\bar\tau - \bar\sm)
  \left\{-e^{2u}\left[\cos{2v}\, g(\f X,\f Y)-\sin{2v}\, g(X,\f Y)\right]\right.\\[4pt]
  &\phantom{=(\bar\tau - \bar\sm)\left\{\right.}\left.
  +e^{2w}\eta(X)\eta(Y)\right\}.
\end{split}
\end{equation}

The substitution $\xi$ for $Y$ in the latter equality gives the following
\begin{equation}\label{ew}
\begin{split}
\D w(\f^2 X)=2e^{w}(\bar\tau - \bar\sm)\eta(X),
\end{split}
\end{equation}
which is valid if and only if
\begin{equation}\label{case1}
\bar\tau = \bar\sm.
\end{equation}
Therefore, \eqref{ew} implies the condition $\D w\circ \f =0$.
In this case, because of \eqref{L-YS}, we obtain $\D u(\xi)=\D v(\xi)=0$, which completes \eqref{F0-YS-uvw}.

Let us remark that $\LL_{\bar\xi} \bg$ vanishes because of \eqref{case1} and \eqref{YS}, i.e. $\bar\xi$ is a Killing vector field in the considered case.

After that, we apply the following relations to \eqref{ff0}, which are equivalent to the last formula in \eqref{cct}:
\begin{equation}\label{gbarg}
\begin{array}{l}
\bar g(\f X,\f Y)=e^{2u}\cos{2v}\,g(\f X,\f Y)-e^{2u}\sin{2v}\,g( X,\f Y),\\[4pt]
\bar g( X,\f Y)=e^{2u}\cos{2v}\,g( X,\f Y)+e^{2u}\sin{2v}\,g(\f X,\f Y).
\end{array}
\end{equation}
Then we use \eqref{F0-YS-uvw} and the formula $\bar\eta=e^{w}\eta$ from \eqref{cct}  to obtain the following expression
\begin{equation}\label{ff0bar}
\begin{aligned}
    \bar{F}(X,Y,Z)=&\ \bg(\f X,\f Y)\,\al(\f^2 Z) + \bg(X,\f Y)\,\bt(\f^2 Z)
\\[4pt]
    &+\bg(\f X,\f Z)\,\al(\f^2 Y) + \bg(X,\f Z)\,\bt(\f^2 Y).
\end{aligned}
\end{equation}
Now, we substitute the Lee forms of $\bar F$ from \eqref{ttbartt} in \eqref{ff0bar} and get
\begin{equation*}\label{ff0bar2}
\begin{aligned}
    \bar{F}(X,Y,Z)=&\ \frac{1}{2n}\Bigl\{\bg(\f X,\f Y)\,\bar\ta(\f^2 Z) + \bg(X,\f Y)\,\bar\ta^*(\f^2 Z)
\\
    &\hspace{20pt}
    +\bg(\f X,\f Z)\,\bar\ta(\f^2 Y) + \bg(X,\f Z)\,\bar\ta^*(\f^2 Y)\Bigr\}.
\end{aligned}
\end{equation*}
The last expression of $\bar F$ means that the obtained manifold belongs to the main class $\F_1$, according to the classification in \cite{GaMiGr}, and the corresponding Lee forms have the following properties:
\[
\bar\ta=-\bar\ta\circ \f^2,\qquad
\bar\ta^*=-\bar\ta^*\circ \f^2,\qquad
\bar\om=0.
\]

Finally, taking into account \eqref{ttbartt}, \eqref{F0-YS-uvw} and the vanishing of $\ta$, $\ta^*$ and $\om$ for any cosymplectic B-metric manifold, we find the expressions of the Lee forms of the transformed manifold as in \eqref{F0F1-YS}.
\end{proof}

As a result from \eqref{F0-YS-uvw}, we obtain that the case in the present section is possible when the functions $(u,v,w)$ of the contact conformal transformation in \eqref{cct} satisfy the conditions:
\begin{itemize}
  \item $u$ and $v$ are constants on the vertical distribution $\mathcal{V}=\Span\xi=\ker\f$;
  \item $w$ is a constant on the horizontal (contact) distribution $\mathcal{H}=\ker\eta=\im\f$.
\end{itemize}

\section{The case when the given manifold is Sasaki-like}

In the present section, we suppose that the given almost contact B-metric manifold $\M$ is Sasaki-like, i.e. the
complex cone $M\times \mathbb R^-$  is a K\"ahler manifold with Norden metric also known as a holomorphic complex Riemannian manifold \cite{IvMaMa45}.

If $\M$ is Sasaki-like, then the following condition is met:
\begin{equation*}\label{Sl}
  \n_X \xi = -\f X.
\end{equation*}
The Sasaki-like condition in terms of $F$ given in \cite{IvMaMa45} is the following
\begin{equation}\label{F=Sl}
F(X,Y,Z)=g(\f X,\f Y)\eta(Z)+g(\f X,\f Z)\eta(Y).
\end{equation}
Then, by virtue of \eqref{FXieta}, \eqref{ff}, \eqref{L1} for $(M,\f,\bar\xi,\bar\eta,\bg)$ and \eqref{F=Sl}, we get
%
\begin{equation}\label{L2uv=}
\begin{split}
  \left(\LL_{\bar\xi} \bg\right)(X,Y) =&\
  2e^{2u-w}
  \left\{\langle\sin{2v}\,\D{u}(\xi)-\cos{2v}\,\left[1-\D v(\xi)\right]\rangle g(X,\f Y)\right.
  \\[4pt]
  &
  \left.-\langle\cos{2v}\, \D{u}(\xi)+\sin{2v}\left[1-\D v(\xi)\right]\rangle g(\f X,\f Y)
  \right\}\\[4pt]
  &
  +e^{w}\left\{\D w(\f^2 X)\eta(Y) +\D w(\f^2 Y)\eta(X)\right\}.
\end{split}
\end{equation}

\begin{theorem}\label{thm:Sl-YS}
An almost contact B-metric manifold that is Sasaki-like can be transformed by a contact conformal transformation of type \eqref{cct} so that the transformed B-metric is a Yamabe soliton with potential the transformed Reeb vector field and a soliton constant $\bar\sm$ if and only if the functions $(u,v,w)$ of the used transformation satisfy the conditions
\begin{equation}\label{L2-YS-uvw}
    \D u(\xi)=0,\qquad \D v(\xi)=1,\qquad \D w=\D w(\xi)\eta.
\end{equation}
Moreover, the obtained Yamabe soliton has a constant scalar curvature with a value $\bar\tau=\bar\sm$ and the transformed almost contact B-metric manifold belongs to a subclass of the main class $\F_1$ determined by:
\begin{equation}\label{ttbartt0=}
    \bar{\ta} = 2n \left\{\D u\circ \f - \D v\circ \f^2\right\},
    \qquad
    \bar{\ta}^* = 2n \left\{\D u - \D v\circ \f\right\}.
\end{equation}
\end{theorem}
\begin{proof}
The expression in \eqref{L2uv=} and the assumption that $\bg$ generates a Yamabe soliton with potential $\bar\xi$ and a soliton constant $\bar\sigma$ on $(M,\f,\bar\xi,\bar\eta,\bg)$ imply the following
\begin{equation}\label{L2-YS}
\begin{split}
&\
    e^{2u-w}
  \left\{\langle\sin{2v}\,\D{u}(\xi)-\cos{2v}\,\left[1-\D v(\xi)\right]\rangle g(X,\f Y)\right.
  \\[4pt]
  &\hspace{35pt}
  \left.-\langle\cos{2v}\, \D{u}(\xi)+\sin{2v}\left[1-\D v(\xi)\right]\rangle g(\f X,\f Y)
  \right\}\\[4pt]
  &
  +\frac12 e^{w}\left\{\D w(\f^2 X)\eta(Y) +\D w(\f^2 Y)\eta(X)\right\}\\[4pt]
  &=(\bar\tau - \bar\sm)
  \left\{e^{2u}\left[\sin{2v}\, g(X,\f Y)-\cos{2v}\, g(\f X,\f Y)\right]\right.\\[4pt]
  &\phantom{=(\bar\tau - \bar\sm)\left\{\right.}\left.
  +e^{2w}\eta(X)\eta(Y)\right\}.
\end{split}
\end{equation}
An obvious consequence for $Y=\xi$ is the following
\begin{equation*}\label{L2-YS-xi}
    \D w(\f^2 X)=2 e^{w}(\bar\tau - \bar\sm)\eta(X),
\end{equation*}
which is satisfied if and only it the following conditions are fulfilled
\begin{gather}
    \bar\tau = \bar\sm,\label{L2-YS-tau}
\\[4pt]
    \D w=\D w(\xi)\eta.\label{L2-YS-w}
\end{gather}

Due to \eqref{L2-YS-tau} and \eqref{YS}, we have the vanishing of $\LL_{\bar\xi} \bg$, which means that $\bar\xi$ is a Killing vector field in this case as well.

Applying \eqref{L2-YS-tau} and \eqref{L2-YS-w} in \eqref{L2-YS}, we obtain
\begin{equation*}\label{L2-YS=0}
\begin{split}
&\
    e^{2u-w}
  \left\{\langle\sin{2v}\,\D{u}(\xi)-\cos{2v}\,\left[1-\D v(\xi)\right]\rangle g(X,\f Y)\right.
  \\[4pt]
  &\hspace{35pt}
  \left.-\langle\cos{2v}\, \D{u}(\xi)+\sin{2v}\left[1-\D v(\xi)\right]\rangle g(\f X,\f Y)
  \right\}=0.
\end{split}
\end{equation*}

The latter equality is valid for arbitrary vector fields if and only if the following conditions are satisfied
\begin{equation*}\label{L2-YS-uv}
    \D u(\xi)=0,\qquad \D v(\xi)=1.
\end{equation*}


Substitute  \eqref{F=Sl} and \eqref{L2-YS-uvw} into \eqref{ff} to get
\begin{equation}\label{FbarSl}
\begin{aligned}
    \bar{F}(X,Y,Z)=&\ \frac{e^{2u}}{2n}\Bigl\{
    \left[\cos{2v}\{\al(Z)-\eta(Z)\}+\sin{2v}\,\bt(Z)\right]
    g(\f X,\f Y)
    \\[4pt]
    &\hspace{24pt}
    +\left[\cos{2v}\,\bt(Z)-\sin{2v}\{\al(Z)-\eta(Z)\}\right]
    g(X,\f Y)
    \\[4pt]
    &\hspace{24pt}
    +\left[\cos{2v}\{\al(Y)-\eta(Y)\}+\sin{2v}\,\bt(Y)\right]
    g(\f X,\f Z)
    \\[4pt]
    &\hspace{24pt}
    +\left[\cos{2v}\,\bt(Y)-\sin{2v}\{\al(Y)-\eta(Y)\}\right]
    g(X,\f Z)\Bigr\}.
\end{aligned}
\end{equation}

Using our assumption that $\M$ is Sasaki-like, we have \cite{IvMaMa45}
\[
\ta=-2n\,\eta,\qquad \ta^*=\om=0.
\]
Then, by  notations \eqref{albt} and conditions \eqref{L2-YS-uvw},  formulae \eqref{ttbartt} yield
\begin{equation}\label{ttbartt0}
    \bar{\ta} = 2n \left[\al - \eta \right],
    \qquad
    \bar{\ta}^* = 2n\, \bt,
    \qquad
    \bar{\om}  = 0.
\end{equation}

Equalities \eqref{albt} and \eqref{L2-YS-uvw} imply the relation
\(
\bt=-\al\circ \f,
\)
which together with the first identity in \eqref{tataom=id} helps to obtain the following expressions
\begin{equation}\label{albtta}
    \al = - \frac{\bar{\ta}\circ\f^2}{2n}+\eta,
    \qquad
    \bt = - \frac{\bar{\ta}\circ\f}{2n}.
\end{equation}
Substitute \eqref{albtta} into  \eqref{FbarSl} and get the formula
\begin{equation*}\label{FbarSlbar}
\begin{aligned}
    \bar{F}(X,Y,Z)=&\ \frac{e^{2u}}{2n}\Bigl\{
    \left[\cos{2v}\,\bar\ta(\f^2Z)+\sin{2v}\,\bar\ta(\f Z)\right]
    g(\f X,\f Y)
    \\[4pt]
    &\hspace{24pt}
    +\left[\cos{2v}\,\bar\ta(\f Z)-\sin{2v}\,\bar\ta(\f^2Z)\right]
    g(X,\f Y)
    \\[4pt]
    &\hspace{24pt}
    +\left[\cos{2v}\,\bar\ta(\f^2Y)+\sin{2v}\,\bar\ta(\f Y)\right]
    g(\f X,\f Z)
    \\[4pt]
    &\hspace{24pt}
    +\left[\cos{2v}\,\bar\ta(\f Y)-\sin{2v}\,\bar\ta(\f^2Y)\right]
    g(X,\f Z)\Bigr\}.
\end{aligned}
\end{equation*}
Then, we apply \eqref{gbarg} in the last equality
and obtain the following expression
\begin{equation*}\label{FbarSl=F1}
\begin{aligned}
    \bar{F}(X,Y,Z)=&\ \frac{1}{2n}\Bigl\{
    \bar\ta(\f^2Z)\bar g(\f X,\f Y)+\bar\ta(\f Z)\bar g(X,\f Y)
    \\[4pt]
    &\hspace{24pt}
    +\bar\ta(\f^2Y)\bar g(\f X,\f Z)+\bar\ta(\f Y)\bar g(X,\f Z)\Bigr\}.
\end{aligned}
\end{equation*}
The obtained form of $\bar F$ coincides with the definition $(M,\f,\bar\xi,\bar\eta,\bar g)$ to belong to the basic class $\F_1$, according to the classification of Ganchev-Mihova-Gribachev in \cite{GaMiGr}.

Finally, the expression of the Lee forms of the transformed manifold given in \eqref{ttbartt0=}  follows from \eqref{ttbartt0} and \eqref{albt}.
\end{proof}

\section{Example}

We recall a known example of a Sasaki-like manifold, given in \cite{IvMaMa45} as Example~2.
A Lie group $G$ of
dimension $5$ is considered to have
 a basis of left-invariant vector fields $\{E_0,\dots, E_{4}\}$
 defined by the following commutators for $\lm,\mu\in\R$:
\[
\begin{array}{ll}
[E_0,E_1] = \lm E_2 + E_3 + \mu E_4,\quad &[E_0,E_2] = - \lm E_1 -
\mu E_3 + E_4,\\[4pt]
[E_0,E_3] = - E_1  - \mu E_2 + \lm E_4,\quad &[E_0,E_4] = \mu E_1
- E_2 - \lm E_3.
\end{array}
\]
An invariant almost contact B-metric structure is then defined
 on $G$  by
\begin{equation*}\label{strEx1}
\begin{array}{rl}
&g(E_0,E_0)=g(E_1,E_1)=g(E_2,E_2)=1\\[4pt]
&g(E_3,E_3)=g(E_4,E_4)=-1,
\\[4pt]
&g(E_i,E_j)=0,\quad
i,j\in\{0,1,2,3,4\},\; i\neq j,
\\[4pt]
&\xi=E_0, \quad \f  E_1=E_3,\quad   \f E_2=E_4.
\end{array}
\end{equation*}
It is verified that the constructed manifold $(G,\f,\xi,\eta,g)$
is an almost contact B-metric manifold that is Sasaki-like.

In \cite{Man62}, the components $R_{ijkl}=R(e_i,e_j,e_k,e_l)$ of its curvature tensor are calculated
and the type of the corresponding Ricci tensor $\rho$ is found. Namely, it is $\rho=4\, \eta\otimes\eta$ and the scalar curvature is $\tau=4$.

Using the non-zero ones of $R_{ijkl}$ determined by the following equalities:
\begin{equation*}\label{Rex1}
\begin{array}{l}
R_{0110}=R_{0220}=-R_{0330}=-R_{0440}=1,\\[4pt]
R_{1234}=R_{1432}=R_{2341}=R_{3412}=1,\qquad    
R_{1331}=R_{2442}=1
\end{array}
\end{equation*}
and the properties $R_{ijkl}=-R_{jikl}=-R_{ijlk}$, we compute that the associated quantity $\tau^*$ of $\tau$ is zero, i.e.
\begin{equation*}\label{tau*ex1}
\tau^*=g^{ij}\rho(E_i,\f E_j)=0.
\end{equation*}

Now, we define the following functions on $\R^5=\left\{(x^0,x^1,x^2,x^3,x^4)\right\}$:
\begin{equation}\label{ExSl=uvw}
\begin{split}
u&=\frac12\ln \left\{\left[(x^1)^2+(x^3)^2\right]\left[(x^2)^2+(x^4)^2\right]\right\},\\[4pt]
v&=\arctan \frac{x^1x^4+x^2x^3}{x^3x^4-x^1x^2}+x^0,
\\[4pt]
w&=x^0.
\end{split}
\end{equation}

The non-zero ones between their partial derivatives are the following:
\begin{equation*}\label{ExSl=d}
\begin{array}{ll}
\ddu{1}=-\ddv{3}=\dfrac{x^1}{(x^1)^2+(x^3)^2},\quad &
\ddu{2}=-\ddv{4}=\dfrac{x^2}{(x^2)^2+(x^4)^2},
\\[9pt]
\ddu{3}=\ddv{1}=\dfrac{x^3}{(x^1)^2+(x^3)^2},\quad &
\ddu{4}=\ddv{2}=\dfrac{x^4}{(x^2)^2+(x^4)^2},
\\[9pt]
\ddv{0}=\ddw{0}=1.
\end{array}
\end{equation*}

Then, for an arbitrary vector field $X=X^i\frac{\p}{\p x^i}$, $i\in\{0,1,2,3,4\}$, we have
\[
\f X=-X^3\frac{\p}{\p x^1}-X^4\frac{\p}{\p x^2}+X^1\frac{\p}{\p x^3}+X^2\frac{\p}{\p x^4},\quad
X^0=\eta(X),\quad \xi=\frac{\p}{\p x^0}.
\]

We verify immediately that the functions defined by \eqref{ExSl=uvw} satisfy the properties
\[
\D u =-\D v\circ \f,\qquad \D u(\xi)=0,\qquad \D v(\xi)=1,\qquad \D w=\eta.
\]

Let us consider a contact conformal transformation defined by \eqref{cct}, where the functions $(u,v,w)$ are determined as in  \eqref{ExSl=uvw}.

Then, the transformed manifold $(G,\f,\bar\xi,\bar\eta,\bar g)$ is an $\F_1$-manifold with a Yamabe soliton with potential $\bar\xi$ and a constant scalar curvature $\bar\tau=\bar\sm$, according ot \thmref{thm:Sl-YS}.
Moreover, taking into account \eqref{ttbartt0=}, we obtain the corresponding Lee forms as follows
\begin{equation*}\label{ttbartt0=ex}
    \bar{\ta} = 4n\, \D u\circ \f,
    \qquad
    \bar{\ta}^* = 4n\, \D u,
    \qquad
    \bar{\om}=0.
\end{equation*}
According to Proposition 8 in \cite{IvMaMa45},
the Ricci tensor of an almost contact B-metric manifold is invariant under a contact homothetic transformation (i.e. when $u$, $v$, $w$ are constants), and therefore for the corresponding scalar curvatures we have
\begin{equation*}\label{t*bar-ex1}
\begin{array}{l}
    \bar\tau = e^{-2u} \cos{2v}\, \tau - e^{-2u} \sin{2v}\, \tau^* + \left\{ e^{-2w} - e^{-2u} \cos{2v}\right\} \rho (\xi,\xi),\\[4pt]
    \bar\tau^* = e^{-2u} \sin{2v}\, \tau + e^{-2u} \cos{2v}\, \tau^* - e^{-2u} \sin{2v}\, \rho (\xi,\xi).
\end{array}
\end{equation*}
Then, bearing in mind the last results, we obtain for our example that
\begin{equation*}\label{t*bar-ex1=}
    \bar\tau = 4\,e^{-2w},\qquad
    \bar\tau^* = 0.
\end{equation*}
Hence, $(G,\f,\bar\xi,\bar\eta,\bar g)$ is scalar flat and the Yamabe soliton constant for $\bg$ is $\bar\sm=4$, and thus the obtained Yamabe soliton is shrinking.

\section*{Acknowledgements}
The author was supported by projects MU21-FMI-008 and FP21-FMI-002 of the Scientific Research Fund,
University of Plovdiv Paisii Hilendarski, Bulgaria.

%
\section*{Conflict of interest}

The author declares that he has no conflict of interest.


\end{document}